\documentclass[a4paper,11pt]{article}
\usepackage[a4paper,text={6.25in,9.0in},centering]{geometry}
\usepackage{amsmath}
\usepackage{amssymb}
\usepackage{amsfonts}
\usepackage{amsthm}
\usepackage{graphicx}
\usepackage{caption}
\usepackage{subcaption}
\usepackage{enumitem}

\usepackage{soul}
\usepackage{xcolor}
\theoremstyle{plain}
\newtheorem{theorem}{\noindent\bf Theorem}[section]
\newtheorem{pro}[theorem]{Proposition}
\newtheorem{corollary}[theorem]{\noindent\bf Corollary}

\numberwithin{equation}{section}
\renewcommand{\theequation}{\arabic{section}.\arabic{equation}}

\theoremstyle{remark}
\newtheorem{definition}[theorem]{\noindent\bf Definition}
\newtheorem{remark}[theorem]{\noindent\bf Remark}
\def\R{\mathbb{R}}
\def\Rg{\mathbb{R}_g}
\def\gi{g^{-1}}
\def\op{\oplus}
\def\om{\ominus}
\def\ot{\otimes}
\def\od{\odot}
\def\oti{\otimes^{-1}}
\def\pwr#1{^{(#1)}}
\def\Ln{{\rm Ln\,}}
\def\l{\lambda}

\def\be{\begin{eqnarray}}
\def\ee{\end{eqnarray}}
\def\ben{\begin{eqnarray*}}
\def\een{\end{eqnarray*}}
\def\benum{\begin{enumerate}}
\def\eenum{\end{enumerate}}

\newcommand\numberthis{\addtocounter{equation}{1}\tag{\theequation}}

\begin{document}
\begin{center}
	\Large\bf{ Classical inequalities for pseudo-integral}
	
\end{center}
\vspace{0.1in}

\begin{center}
	{\bf Pankaj Jain}\\
	Department of Mathematics\\
	South Asian University\\
	Akbar Bhawan, Chanakya Puri, New Delhi-110021, India\\
	(email: pankaj.jain@sau.ac.in, pankajkrjain@hotmail.com)
	\medskip
	
\end{center}
\bigskip

\begin{abstract}
In this paper, we have derived certain classical inequalities, namely, Young's, H\"older's, Minkowski's and Hermite-Hadamard inequalities for pseudo-integral  (also known as $g$-integral). For Young's, H\"older's, Minkowski's inequalities, both the cases $p>1$ as well as $p<1,\,p\ne 0$ have been covered. Moreover, in the case of Hermite-Hadamard inequality, a refinement has also been proved and as a special case, $g$-analogue of geometric-logarithmic-arithmetic inequality has been deduced. 
\end{abstract}

\medskip


\section{Introduction}
There has been a continued interest in the theory of fuzzy measures ever since Sugeno \cite{sug} defined this measure having monotonicity property instead of additivity. This notion has shown many applications in the theory of fuzzy sets defined by Zadeh \cite{zadeh}. Later on, several other measures were defined and studied which are non-additive and have some advantages over each other. We refer to \cite{cho}, \cite{ich}, \cite{mas}, \cite{weber1}, \cite{weber2} and references therein for such consideration.

In \cite{pap2} (see also \cite{pap1}, \cite{pap3}, \cite{pap4}, \cite{pap5}, \cite{pap6}, \cite{flores}), Pap initiated the so called $\op$-decomposable measures based on the pseudo algebraic operations, i.e., pseudo-addition $\op$ and pseudo-multiplication $\ot$. Consequently, in this framework, an integral and the derivative have been defined. Since the pseudo operations involve a generator function, usually denoted by $g$, the corresponding notions are also refered to as $g$-addition, $g$-multiplication, $g$-derivative, $g$-integral etc.

In the present paper, we make an attempt to contribute in $g$-calculus by deriving several classical inequalities in this framework, namely, Young's inequality, H\"older's inequality, Minkowski's inequality and Hermite-Hadamard inequality. Let us point out that for the case $1<p<\infty$, $g$-analogue of H\"older's and Minkowski's inequalities were derived in \cite{agahi} by using the corresponding classical inequalities. In our case, we derive these inequalities by using $g$-Young's inequality that we first prove in this paper. Moreover, we also establish certain other variants of H\"older's inequality and we cover the case $p<1,\, p\ne 0$ as well.

The paper is organized as follows. In Section 2, we collect the entire ``pseudo-machinery'' that is required throughout the paper. Here apart from the known notions and concepts, we define pseudo-logarithmic function. We also define pseudo-absolute value which helps us to establish the inequalities in a wider domain. Section 3, starts with Young's inequality followed by H\"older's inequality and its variants and finally we prove Minkowski's inequality in this section. In Section 4, we define pseudo-convexity and in this context prove the Hermite-Hadamard inequality. As a special case, the $g$-analogue of geometric-logarithmic-arithmetic mean inequality is obtained. Also in this section, we prove a refined version of Hermite-Hadamard inequality. Finally, in Section 5, we summarize the work done in this paper and make some suggestions for future work.

\section{The Pseudo Setting}
In this section, we collect some basic algebraic operations, elementary functions, derivative and integral in the framework of pseudo algebra. Most of these notions are already known, however, we define absolute value and logarithm formally here.

\subsection{The algebraic operations}

Let $\op$ denote the pseudo-addition, a function $\op:\R\times\R\to\R$ which is commutative, nondecreasing in each component, associative and has a zero element. We shall assume that $\op$ is strict pseudo-addition which means that $\op$ is strictly increasing in each component and continuous. By Aczel's Theorem \cite{ling}, for each strict pseudo-addition, there exists a monotone function $g$ (called the generator for $\op$), $g:\R\to\R^+$ such that
\begin{equation*}
	x\op y = \gi\left(g(x) + g(y)\right).
\end{equation*}
Similarly, we denote by $\ot$, the pseudo-multiplication which is a function $\op:\R\times\R\to\R$ which is commutative, nondecreasing in each component, associative and has a unit element. The operation $\ot$ is defined by
\begin{equation*}
	x\ot y = \gi\left(g(x) \cdot g(y)\right).
\end{equation*}
Now onwards, the set $\R$ equipped with the pseudo operations $\op$ and $\ot$ with the corresponding generator $g$ will be denoted by $\Rg$. The zero and the unit elements of $\Rg$  will be denoted, respectively, by $0_g$ and $1_g$. 

In \cite{bc}, the authors considered that the generator function $g:\R_g\to\R$ is strictly monotone (either strictly increasing or strictly decreasing), onto, $g(0_g)=0$, $g'(x)\ne 0$, for all $x$, $g\in C^2$ and $\gi\in C^2$. Using this map, the following well defined operations have been defined:
\begin{equation*}
x\om y =	\gi\left(g(x) - g(y)\right), \quad x\oti y=\gi\left( \frac{g(x)}{g(y)}\right),\,\,{\rm provided}\,\, y\ne 0_g.
\end{equation*} 

The order relation in $\Rg$, denoted by $\le_g$, satisfies the following:
\begin{equation*}
	x\le_g y \iff x\om y \le 0_g.
\end{equation*}
If $x\le_g y$, we can also write it as $y\ge_g x$. If $x\le_g y$ and $x\ne y$, we shall write it as $x<_gy$, or equivalently, $y>_gx$.

In order to make $\Rg$ a linear space over the field $\R$, we define the pseudo-scalar product:
\begin{equation*}
	n\odot x = \gi(ng(x)), \quad x\in \Rg,\,\, n\in \R.
\end{equation*}
It was pointed out in \cite{bc} that the operations $\odot$ and $\ot$ are different. There was a need to define the scalar product $\odot$ since the compatibility condition $1\odot x=x$ is not satisfied by $\ot$.

\begin{remark}
$(\Rg,\op,\ot,\le_g)$ is an ordered and complete algebra.
\end{remark}

\subsection{Differentiation and integration}

The pseudo-derivative or more commonly called $g$-derivative of a suitable function $f:[a,b]\subseteq\R\to\R_g$ is defined by
\begin{equation*}
	D^\op f(x):=\frac{d^\op f(x)}{dx}=\gi\left( (g\circ f)'(x)   \right).
\end{equation*}

In \cite{mar} (see also \cite{bc}), a more general $g$-derivative was defined as
\begin{equation*}
D_g f(x):=	\frac{d_g f(x)}{dx}=\lim_{h\to 0}\left[ f(x\op h)\om f(x)  \right]\oti h
\end{equation*}
and it was shown that
\begin{equation*}
D_g f(x)=\gi\Big((g\circ f)'(x)/g'(x)   \Big).
\end{equation*}
	
The pseudo-integral or the $g$-integral of a suitable function $f:[a,b]\subseteq \R\to\Rg$ is defined by
\begin{equation}\label{e2.1}
\int^\op_{[a,b]}f(t)\ot dt=\gi\left(  \int_a^b (g\circ f)(x)\,dx  \right).
\end{equation}
Similar to the notion of the $g$-derivative, in \cite{mar}, a more general $g$-integral was defined which is given by
\begin{equation*}
	\int^g_{[a,b]}f(t)\ot dt=\gi\left(  \int_a^b (g\circ f)(x)g'(x)\,dx  \right).
\end{equation*}

\begin{remark}
	The $g$-derivative $D^\op$ and the $g$-integral $\displaystyle\int^\op_{[a,b]}$ can be obtained as special cases of, respectively, $D_g$ and $\displaystyle\int^g_{[a,b]}$. Since the justification requires notions which are beyond the scope of this paper, we refer to  \cite{bc}, \cite{mar} for details. Unless specified otherwise, in this paper, the $g$-derivative and $g$-integral will be used as represented by $D_g$ and $\displaystyle\int^g_{[a,b]}$, respectively.
\end{remark}

Some of the properties of $g$-integral are mentioned below:
\begin{itemize}
	\item [(a)] $\displaystyle \int^g_{[a,b]}(f\op h)\ot dt=\int^g_{[a,b]}f\ot dt \op \int^g_{[a,b]} h\ot dt$
	
	\item [(b)] $\displaystyle \int^g_{[a,b]}(\lambda\ot f)\ot dt=\lambda\ot\int^g_{[a,b]}f\ot dt $
	
	\item [(c)] $\displaystyle \int^g_{[a,b]}(\lambda\od f)\ot dt=\lambda\od\int^g_{[a,b]}f\ot dt $
	
	\item [(d)] $\displaystyle f\le_g h\Rightarrow \int^g_{[a,b]}f\ot dt \le_g \int^g_{[a,b]} h\ot dt$
\end{itemize}
\subsection{The exponent}
Define the set
\begin{equation*}
	\Rg^+=\{x\in\Rg : 0_g\le_g x\}.
\end{equation*}
In view of the operation $\ot$, for $x\in\Rg$ and $n\in\mathbb N$, we define
\begin{equation*}
	x\pwr n:=\underbrace{x\ot x\ot\hdots \ot x}_{n-{\rm times}}=\gi(g^n(x)).
\end{equation*}
This notion of exponent can be extended for a general $p\in (0,\infty)$ which is defined (see \cite{agahi}, \cite{ich}) for all $x\in \Rg^+$ as
\begin{equation*}
	x\pwr p=\gi(g^p(x)).
\end{equation*}
It can further be generalized to cover negative powers as well: For $p\in (0,\infty)$, we define
\begin{equation*}
	x\pwr {-p}=1_g\oti x\pwr p, \quad x\in \Rg^+.
\end{equation*}
For $p=0$, we define $x\pwr 0=1_g$.
It is easy to check that for $p,q\in (0,\infty)$ and $x\in\Rg^+$, the following laws of exponent hold:
\begin{enumerate}
\item [(i)]	$x\pwr p \ot x\pwr q = x\pwr {p+q}$
\item [(ii)] $\left( x\pwr p  \right)\pwr q=x\pwr {pq}$
\item [(iii)]	$(x\ot y) \pwr p =x\pwr p \ot y\pwr p$
\item [(iv)]	$(x\oti y) \pwr p =x\pwr p \oti y\pwr p$
\item [(v)]	$(\alpha\od x) \pwr p =\alpha^ p \od x\pwr p$
\end{enumerate}

\subsection{Absolute value}

For $x\in\Rg$, we define its $g$-absolute value as follows:
\begin{equation*}
	|x|_g:=
	\begin{cases}
		x,&{\rm if}\, 0_g\le_g x\\
		-x,& {\rm if}\, x<_g 0_g.
	\end{cases}
\end{equation*}
It can be seem that
\begin{equation*}
	|x|_g=\gi\left(|g(x)| \right).
\end{equation*}
Note that
$$
|x\op y|_g\le_g |x|_g \op |y|_g
$$
if $g$ is increasing and the inequality is reversed if $g$ is decreasing.

\subsection{Exponential and logarithm functions}

The $g$-exponential function for $x\in\Rg$ as defined in \cite{bc} is given by
\begin{equation*}
	E\pwr x=\gi(e^{g(x)}),
\end{equation*}
where $e^{g(x)}$ is the standard exponential function. It is natural to define $g$-logarithm function  by
\begin{equation*}
	\Ln x=\gi\left(\ln{g(x)}\right),
\end{equation*}
where $\ln g(x)$ is the standard logarithm function. 

\begin{remark}
	Unlike in the standard case, we may have that $E\pwr x<_g 0_g$. In fact, if the generator $g$ is monotonically decreasing, then for any $x\in \Rg$, $E\pwr x\le_g0$, since $\gi$ is also monotonically decreasing. This suggests that in the pseudo case, logarithm can be defined for "negative" numbers, which in fact is true if, again, the generator $g$ is decreasing. 
\end{remark}

Through the following proposition, we provide several properties of $E\pwr x$ and $\Ln x$, the proofs of which can be worked out easily.

\begin{pro}\label{pr2.4}	The following hold:
	\begin{enumerate}[label=\rm(\roman*)]
	\item $D_g(E\pwr x)=E\pwr x$
	
	\item $\displaystyle\int^g E\pwr x \ot dx =E\pwr x$
	
	\item $E\pwr x\ot E\pwr y=E\pwr{x\op y}$
	
	\item $E\pwr {\Ln x}=x$
	
	\item $\Ln E\pwr x=x$
	
	\item $\Ln (x\ot y)=\Ln x\op \Ln y$
	
	\end{enumerate}
	
\end{pro}

\section {Inequalities}

In this section, we shall prove $g$-analogue of Young's, H\"older's and Minkowski's inequalities. Here and throughout, for any $p\in\R,\,p\ne 0$, $p'$ will denote the conjugate index to $p$, i.e., ${1\over p} + {1\over {p'}}=1$.

\subsection {Young's inequality}

The classical Young's inequality asserts that for $1<p<\infty$ and $a,b>0$, it holds:
$$
ab \le \frac{a^p}{p} + \frac{b^{p'}}{p'}
$$
whereas the inequality gets reversed if $p<1,\,p\ne 0$. We prove the $g$-analogue of this inequality below:

\begin{theorem}\label{t-y1}
Let $1<p<\infty$.
\begin{enumerate}
\item [\rm (a)] If the generator $g$ is increasing then for all $a,b\in \Rg^+$, the following Young's type inequality holds:
\begin{equation}\label{e3.1}
	a\ot b \le_g \left( a\pwr p \oti\gi(p) \right) \op \left( b\pwr {p'} \oti\gi(p') \right).
\end{equation}

\item [\rm (b)] If the generator $g$ is decreasing then for all $a,b\in \Rg^+$, the  inequality \eqref{e3.1} holds in the reverse direction.
\end{enumerate}
\end{theorem}

\begin{proof}
(a) In view of Proposition \ref{pr2.4}, we have
\begin{align}\label{e3.2}
	a\ot b &=E^{\Ln (a\ot b)}\nonumber\\
	&=E^{(\Ln a\op \Ln b)}.
\end{align}
Note that for any $1<p<\infty$
\begin{align*}
\Ln a\pwr p &=\gi\left( \ln g^p(a)  \right)\\
&=\gi(p\ln g(a))\\
&=\gi\Big(g(\gi(p))\cdot g(\gi(\ln g(a)))  \Big)\\
&=\gi(p)\ot \gi(\ln g(a))\\
&=\gi(p)\ot \Ln a
\end{align*}
which gives that
\begin{equation*}
\Ln a\pwr p \oti \gi(p)=\Ln a.
\end{equation*}
Thus
\begin{align}\label{e3.4}
\Ln a &= \Ln a\pwr p \oti \gi(p)\nonumber\\
&=\gi\Big(\ln g^p(a)\Big) \oti \gi(p)\nonumber\\
&=\gi\left(\frac{\ln g^p(a)}{p}\right).
\end{align}
Similarly, since $1<p'<\infty$, we have	
\begin{equation}\label{e3.5}
\Ln b = \gi\left(\frac{\ln g^{p'}(b)}{p'}\right).
\end{equation}

By using \eqref{e3.4} and \eqref{e3.5} in \eqref{e3.2}, using a known inequality and increasingness of $\gi$ (since $g$ is so), we get
\begin{align}
a\ot b&=E^{\left(\gi\big(\frac{\ln g^p(a)}{p}\big)\op  \gi\big(\frac{\ln g^{p'}(b)}{p'}\big) \right)	}\nonumber\\
&=E^{\left(\gi {\left({1\over p}\ln g^p(a) + {1\over p'}\ln g^{p'}(b)    \right)} \right)	}\nonumber\\
&=\gi \left( e^{\left({1\over p}\ln g^p(a) + {1\over p'}\ln g^{p'}(b)    \right)} \right)\nonumber\\
&\le_g \gi {\left({1\over p}e^{\ln g^p(a)} + {1\over p'}e^{\ln g^{p'}(b)}    \right)}\label{e3.6}\\
&=\gi {\left({1\over p}{g^p(a)} + {1\over p'}{g^{p'}(b)}    \right)}=:A.\label{e3.7}
\end{align}

Further, we find that
$$
a\pwr p \oti\gi(p)=\gi(g^p(a)) \oti\gi(p) = \gi\left(\frac{g^p(a)}{p}\right)
$$
and similarly
$$
b\pwr {p'} \oti\gi(p')=\gi(g^p(a)) \oti\gi(p) = \gi\left(\frac{g^{p'}(b)}{p'}\right)
$$
so that
\begin{equation}\label{e3.8}
a\pwr p \oti\gi(p) \op 	b\pwr {p'} \oti\gi(p') = 
\gi\left(\frac{g^p(a)}{p} +  \frac{g^{p'}(b)}{p'}  \right)=:B.
\end{equation}
The assertion now follows in view of \eqref{e3.7} and \eqref{e3.8} since $A\om B= 0_g$.
\medskip

(b) Since the generator $g$ is decreasing and consequently $\gi$ is so,  the inequality \eqref{e3.6} gets reversed and the assertion follows. 
\end{proof}

Below we prove $g$-Young's inequality for the case $p<1,\,p\ne 0$.

\begin{theorem}\label{t-y2}
	Let $p<1,\,p\ne 0$. 
	\begin{enumerate}
		\item [\rm (a)] If the generator $g$ is increasing then for all $a,b\in \Rg^+$, the following Young's type inequality holds:
		\begin{equation}\label{e-y1}
			a\ot b \ge_g \left( a\pwr p \oti\gi(p) \right) \op \left( b\pwr {p'} \oti\gi(p') \right).
		\end{equation}
		
		\item [\rm (b)] If the generator $g$ is decreasing then for all $a,b\in \Rg^+$, the  inequality \eqref{e-y1} holds in the reverse direction.
	\end{enumerate}
\end{theorem}

\begin{proof}
(a) Without any loss of generality, we assume that $0<p<1$ so that $p'<0$ for otherwise we can interchange the roles of $p$ and $p'$.

Take $r=1/p$ and $s=-p'/p$. Then $1<r,s<\infty$ and $1/r + 1/s = 1$. Set
$$
x=(a\ot b)\pwr p\quad {\rm and}\quad y=b\pwr {-p}.
$$
We apply inequality \eqref{e3.1} on $x,y$ and with exponents $r,s$ and obtain
\begin{equation}\label{e-y3}
	x\ot y \le_g \left( x\pwr r \oti\gi(r) \right) \op \left( y\pwr {s} \oti\gi(s) \right).
\end{equation}
Now, it can be calculated that
$$
x\ot y=a\pwr p,\quad x\pwr r \oti\gi(r)=(a\ot b)\ot \gi(\frac{1}{p}),\quad y\pwr {s} \oti\gi(s) = b\pwr {p'}\oti \gi (\frac{-p'}{p})
$$
which on substituting in \eqref{e-y3} and rearranging the terms give the result.
\medskip

(b) This can be obtained using the similar arguments and applying Theorem \ref{t-y1}(b).
\end{proof}

\subsection{H\"older's inequality}

Let $1<p<\infty$. We denote by $L_p(\R)$, the Lebesgue space which consists of all measurable functions defined on $\R$ such that
$$
\|f\|_{p,\R}=\left(\int_\R |f(x)|^p\, dx\right)^{1/p}<\infty.
$$
By a weight function, we mean a function which is measurable, positive and finite almost everywhere (a.e) on $\R$. For a weight function $w$, we denote by $L_{p,w}(\R)$, the weighted Lebesgue space which consists of all measurable functions defined on $\R$ such that
$$
\|f\|_{p,w,\R}=\left(\int_\R |f(x)|^p w(x)\, dx\right)^{1/p}<\infty.
$$
The spaces $L_p(\R)$ and $L_{p,w}(\R)$ are both Banach spaces.

Let $f:\R\to \Rg$ be a mesurable function and $p\in\R,\, p\ne 0$. We define
\begin{equation}\label{l1}
	[f]_{p,\Rg}^g:=\left(\int_{\Rg}^g |f(x)|_g\pwr p\ot dx\right)\pwr {{1\over p}}\quad {\rm and}\quad [f]_{p,g,\R}:=\left(\int_{\R} |f(x)|^pg'(x)\,dx\right)^{{1\over p}}.
\end{equation}

\begin{remark}\label{r-3.3}
	The expression $[f]_{p,g,\R}$ in \eqref{l1} is not a norm of some weighted Lebesgue space unless $1<p<\infty$ and  $g'$ qualifies to be a weight function which can be a case if, e.g., $g$ is an increasing function and consequently $g'>0$. In such a case, $[f]_{p,g,\R}$ will enjoy all the properties of a norm.
	
\end{remark}

As an independent interest, it is easy to observe that the expressions $[f]_{p,\R_g}^g$ and  $[f]_{p,g,\R}$ are connected. Precisely, we prove the following:

\begin{pro}\label{t-l1}
Let $f:\R\to \Rg$ be a measurable function and $p\in\R,\, p\ne 0$. Then
$$
[f]_{p,\Rg}^g = \gi\Big([g\circ f]_{p,g,\R} \Big).
$$	
\end{pro}

\begin{proof}
Let $p>0$. We have
$$
|f(x)|_g\pwr p = \Big(\gi(|g(f(x))|)\Big)\pwr p = \gi\Big(|g(f(x))|^p\Big)
$$
so that
$$
\int_{\Rg}^g |f(x)|_g\pwr p\ot dx = \gi\left(\int_\R |g(f(x))|^p g'(x)\, dx\right)
$$
which gives that
$$
[f]_{p,\Rg}^g = \gi\left(\left(\int_\R |g(f(x))|^p g'(x)\, dx\right)^{1/p} \right)
$$
and we are done in this case. Similar arguments can be employed to prove the assertion for $p<0$.
\end{proof}

An element $x\in\R_g$ is said to be finite, written $x<_g\infty$, if there exists $c\in\R_g$ such that $x<_g c$.

We prove the following H\"older's type inequality:

\begin{theorem}\label{t-h1}
Let $1<p<\infty$. 
\begin{enumerate}
	\item [\rm (a)] Let the generator $g$ be increasing and  $f,h:\R\to \Rg$ be measurable functions such that $[f]_{p,\Rg}^g<_g\infty$ and $[h]_{p',\Rg}^g<_g\infty$. Then $[f\ot h]^g_{1,\Rg}<_g\infty$ and the following H\"older's type inequality holds:
	\begin{equation}\label{e3.9}
		[f\ot h]^g_{1,\Rg}\le_g [f]^g_{p,\Rg}\ot [h]^g_{p',\Rg}.
	\end{equation}
	
	\item [\rm (b)] If the generator $g$ is decreasing then  the  inequality \eqref{e3.9} holds in the reverse direction.
\end{enumerate}	
\end{theorem}

\begin{proof}
(a) We shall be using Theorem \ref{t-y1} for appropriate $a,b\in\Rg^+$. Choose
\begin{equation}\label{e3.10}
	a=|f(x)|_g \oti [f]^g_{p,\Rg}
\end{equation}
and
\begin{equation}\label{e3.11}
	b=|h(x)|_g \oti [h]^g_{p',\Rg}.
\end{equation}	

We have
\begin{equation}\label{e3.12}
\int_{\Rg}^g a\ot b \ot dx = \left( \int_{\Rg}^g |f(x)\ot h(x)|_g\ot dx \right)	\oti \left( [f]^g_{p,\Rg}\ot [h]^g_{p',\Rg} \right).
\end{equation}

Further, it can be seen that
\begin{equation*}
a\pwr p = |f(x)|_g\pwr p \oti \Big([f]^g_{p,\Rg}\Big)\pwr p
\end{equation*}
so that
\begin{equation*}
	{a\pwr p}\oti {\gi(p)} =\left\{ |f(x)|_g\pwr p \oti \Big([f]^g_{p,\Rg}\Big)\pwr p \right\} \oti {\gi(p)}
\end{equation*}
which on $g$-integrating gives
\begin{align*}
\int_{\Rg}^g {a\pwr p}\oti {\gi(p)} \ot dx &= \left\{\left(\int_{\Rg}^g |f(x)|_g\pwr p \ot dx \right)\oti [f]^g_{p,\Rg} \right\} \oti {\gi(p)}\\
&=1_g\oti \gi(p)\\
&=\gi\Big(\frac{g(1_g)}{p}\Big).
\end{align*}
Similarly, one can obtain that
\begin{equation*}
\int_{\Rg}^g {b\pwr {p'}}\oti {\gi(p')} \ot dx	= \gi\Big(\frac{g(1_g)}{p'}\Big)
\end{equation*}
which together with the last equation gives
\begin{align} \label{e3.13}
\left(\int_{\Rg}^g {a\pwr p}\oti {\gi(p)} \ot dx\right)	
\op \left(\int_{\Rg}^g {b\pwr {p'}}\oti {\gi(p')} \ot dx\right)
&= \gi\Big(\frac{g(1_g)}{p}\Big) \op \gi\Big(\frac{g(1_g)}{p'}\Big)\nonumber\\
&=\gi\Big(\frac{g(1_g)}{p}+ \frac{g(1_g)}{p'}\Big)\nonumber\\
&=\gi(g(1_g))\nonumber\\
&=1_g.
\end{align}
Now, the inequality \eqref{e3.9} follows in view of \eqref{e3.12} and \eqref{e3.13} if we take $a$ and $b$ given, respectively, by \eqref{e3.10} and \eqref{e3.11} in Theorem \ref{t-y1} and take $g$-integral on both the sides of the resulting inequality.

(b) This follows since the inequality \eqref{e3.1} gets reversed for decreasing $g$.
\end{proof}

A generalized version of Theorem \ref{t-h1} is the following:

\begin{theorem}\label{t-h2}
	Let $1<p,q,r<\infty$ be such that ${1\over p}+{1\over q} = {1\over r}$. 
	\begin{enumerate}
		\item [\rm (a)] Let the generator $g$ be increasing and  $f,h:\R\to \Rg$ be measurable functions such that $[f]_{p,\Rg}^g<_g\infty$ and $[h]_{p',\Rg}^g<_g\infty$. Then $[f\ot h]^g_{r,\Rg}<_g\infty$ and the following H\"older's type inequality holds:
		\begin{equation}\label{e3.14}
			[f\ot h]^g_{r,\Rg}\le_g [f]^g_{p,\Rg}\ot [h]^g_{q,\Rg}.
		\end{equation}
		
		\item [\rm (b)] If the generator $g$ is decreasing then the  inequality \eqref{e3.14} holds in the reverse direction.
	\end{enumerate}	
\end{theorem}	

\begin{proof} (a)
Write $P=p/r$, $Q=q/r$ so that ${1\over P}+{1\over Q} = 1$. Applying the inequality \eqref{e3.9} for the functions $F:=f\pwr r,\,H:=h\pwr r$ and with the exponents $P,Q$, we obtain
\begin{equation}\label{e3.15}
	[F\ot H]^g_{1,\Rg}\le_g [F]^g_{P,\Rg}\ot [H]^g_{Q,\Rg}.
\end{equation}
It can be calculated that
\begin{align*}
[F]^g_{P,\Rg} &= \left( [f]^g_{p,\Rg} \right)\pwr{r/p},\\	
[H]^g_{Q,\Rg} &= \left( [h]^g_{q,\Rg} \right)\pwr{r/q}\\
\end{align*}
and
\begin{equation*}
	[F\ot H]^g_{1,\Rg}=\left( [f\ot h]^g_{r,\Rg}\right)\pwr{r}
\end{equation*}
using which in \eqref{e3.15} and adjusting the exponents, \eqref{e3.14} follows.
\medskip

(b) This is an immediate consequence of the fact that $g$ is decreasing.
\end{proof}

As a consequence of Theorem \ref{t-h2}, we prove the following:
	
\begin{theorem}
	Let $1<p,q,r<\infty$ be such that ${t\over p}+{{1-t}\over q} = {1\over r}$, where $0<t<1$.
	\begin{enumerate}
		\item [\rm (a)] If the generator $g$ is increasing and for all measurable functions $f:\R\to \Rg$,
	$$
	[f]^g_{p,\Rg}<_g \infty \quad {and}\quad [f]^g_{q,\Rg}<_g \infty,
	$$	
then $[f]^g_{r,\Rg}<_g \infty$ and the following inequality holds:
		\begin{equation}\label{e3.16}
			[f]^g_{r,\Rg}\le_g \Big([f]^g_{p,\Rg}\Big)\pwr t \ot \Big([f]^g_{q,\Rg}\Big)\pwr {1-t}.
		\end{equation}
		
		\item [\rm (b)] If the generator $g$ is decreasing then the  inequality \eqref{e3.16} holds in the reverse direction.
	\end{enumerate}	
\end{theorem}

\begin{proof}
(a) Applying Theorem \ref{t-h2} for the functions $f\pwr t,\,f\pwr{1-t}$ and with the exponents $p\over t$, $q\over{1-t}$, we obtain that
$$
[f\pwr t\ot f\pwr{1-t}]^g_{r,\Rg}\le_g [f\pwr t]^g_{p/t,\Rg}\ot [f\pwr {1-t}]^g_{q/{1-t},\Rg}.
$$
Now, since it can be worked out that
$$
[f\pwr t]^g_{p/t,\Rg}=\Big([f]^g_{p,\Rg}\Big)\pwr t
$$
and
$$
[f\pwr {1-t}]^g_{q/{1-t},\Rg}=\Big([f]^g_{q,\Rg}\Big)\pwr {1-t},
$$
the assertion follows.
\medskip

(b) This is an immediate consequence of the fact that $g$ is decreasing.
\end{proof}

Theorem \ref{t-h1} provides $g$-H\"older's inequality for $1<p<\infty$. Here, we used $g$-Young's ineuality for the same range of $p$ given in Theorem \ref{t-y1}. On the similar lines, by using Theorem \ref{t-y2}, $g$-H\"older's inequality for $p<1,\, p\ne 0$ can be obtained. We only state the theorem below:

\begin{theorem}\label{t-rh}
	Let $0<p<1$. 
	\begin{enumerate}
		\item [\rm (a)] If the generator $g$ is increasing then for all measurable functions $f,h:\R\to \Rg$, the following  inequality holds:
		\begin{equation}\label{e3.17}
			[f\ot h]^g_{1,\Rg}\ge_g [f]^g_{p,\Rg}\ot [h]^g_{p',\Rg}.
		\end{equation}
		
		\item [\rm (b)] If the generator $g$ is decreasing then  the  inequality \eqref{e3.17} holds in the reverse direction.
	\end{enumerate}	
\end{theorem}

\begin{remark}
Theorem \ref{t-h1} is a generalization of $g$-H\"older's inequality proved in \cite{agahi}  for the case $1<p<\infty$. There the authors used the integral $\displaystyle\int^{\op}_{[a,b]}$ as defined in \eqref{e2.1} which is a special case of the integral $\displaystyle\int^g_{[a,b]}$ used in this paper. Moreover our proof is based on $g$-Young's inequality and we have covered the case $p<1,\,p\ne 0$ as well. The other variants of $g$-H\"older's inequality that we prove in this subsection are also new.
\end{remark}

\subsection{Minkowski's inequality}

\begin{theorem}\label{t-m1}
	Let $1<p<\infty$. 
	\begin{enumerate}
		\item [\rm (a)] If the generator $g$ is increasing then for all measurable functions $f,h:\R\to \Rg$, the following Minkowski's type inequality holds:
		\begin{equation}\label{e-m1}
			[f\op h]^g_{p,\Rg}\le_g [f]^g_{p,\Rg}\op [h]^g_{p,\Rg}.
		\end{equation}
		
		\item [\rm (b)] If the generator $g$ is decreasing then  the  inequality \eqref{e-m1} holds in the reverse direction.
	\end{enumerate}	
\end{theorem}

\begin{proof}(a)
	Since $g$ is increasing, we have
\begin{equation}\label{e-m2}
|f\op h|_g\le_g |f|_g \op |h|_g
\end{equation}
and consequently
\begin{align}\label{e-m3}
\int_{\Rg}^g |f\op h|_g\pwr p\ot dx	&=\int_{\Rg}^g |f\op h|_g\pwr {p-1} \ot |f\op h|_g\ot dx\nonumber\\
&\le_g \int_{\Rg}^g |f\op h|_g\pwr {p-1} \ot |f|_g\ot dx \op \int_{\Rg}^g |f\op h|_g\pwr {p-1} \ot |h|_g\ot dx\nonumber\\
&=: I_1\op I_2.
\end{align}
We apply H\"older's inequality \eqref{e3.9} and obtain
\begin{align*}
I_1 &\le_g 	[(f\op h)\pwr {p-1}]^g_{p',\Rg}\ot [f]^g_{p,\Rg}\\
&=\left([f\op h]^g_{p,\Rg}\right)^{(p)\oti (p')}\ot [f]^g_{p,\Rg}.
\end{align*}
Similarly
\begin{align*}
	I_2 
	&\le_g\left([f\op h]^g_{p,\Rg}\right)^{(p)\oti (p')}\ot [h]^g_{p,\Rg}
\end{align*}	
so that \eqref{e-m3} gives
$$
\int_{\Rg}^g |f\op h|_g\pwr p\ot dx\le_g \left([f]^g_{p,\Rg} \op [h]^g_{p,\Rg} \right) \ot \left([f\op h]^g_{p,\Rg}\right)^{(p)\oti (p')}.
$$
Now, adjusting the powers of the factor $[f\op h]^g_{p,\Rg}$, the assertion follows.
\medskip

(b) This follows immediately since the inequality \eqref{e-m2} and H\"older's inequality both hold in the reverse direction in this case.
\end{proof}

The case of the Minkowski inequality when $p<1,\, p\ne 0$ can be discussed similarly. We only state the result below:

\begin{theorem}\label{t-m2}
	Let $0<p<1,\, p\ne 0$. 
	\begin{enumerate}
		\item [\rm (a)] If the generator $g$ is increasing then for all measurable functions $f:\R\to \Rg$, the following Minkowski's type inequality holds:
		\begin{equation}\label{e-m4}
			[f\op h]^g_{p,\Rg}\ge_g [f]^g_{p,\Rg}\op [h]^g_{p,\Rg}.
		\end{equation}
		
		\item [\rm (b)] If the generator $g$ is decreasing then  the  inequality \eqref{e-m4} holds in the reverse direction.
	\end{enumerate}	
\end{theorem}

\section{Hermite-Hadamard Inequality}

In this section, we shall consider the special case of the integral $\displaystyle\int^g_{[a,b]} f(x)\ot dx$, i.e., $\displaystyle\int^\op_{[a,b]} f(x)\ot dx$ given by \eqref{e2.1}.

The classical Hermite-Hadamard inequality asserts that if $f:[a,b]\to \R$ is a convex function then the following holds:
$$
f\left({a+b\over 2}\right) \le { 1\over{b-a}} \int_a^b f(x)\, dx \le \frac{f(a)+f(b)}{2}
$$
and the inequalities are reversed if $f$ is concave. The aim of this section is to derive a $g$-analogue of this inequality. First, we define the following:

\begin{definition}
	A function $f:[a,b]\to \Rg$ is said to be pseudo-convex on $[a,b]$ if for all $x,y\in[a,b]$ and all $0\le\lambda\le 1$
	$$
	f(\l x + (1-\l)y)\le_g \l\od f(x)\op (1-\l)\od f(y).
	$$
\end{definition}

We shall denote
$$
\sigma :=\l b + (1-\l) a\quad {\rm and}\quad \delta:=\l a + (1-\l)b.
$$

\begin{theorem}\label{t-hh1}
Let $f:[a,b]\to \Rg$ be a pseudo-convex function. Then the following Hermite-Hadamard inequality holds:
\begin{equation}\label{e-hh1}
	f\left({a+b\over 2}\right) \le_g \left( 1\over{b-a}\right)\od \int^\op_{[a,b]} f(x)\ot dx \le_g {1\over 2}\od \big(f(a)\op f(b)\big).
\end{equation}	
	
\end{theorem}

\begin{proof}
Since $f$ is pseudo-convex, we have
\begin{align*}
	f\left(\sigma + \delta \over 2\right) &\le_g {1\over 2}\od (f(\sigma) \op f(\delta))\\
	&\le_g {1\over 2}\od \Big( \l\od f(b)\op (1-\l)\od f(a)\op \l\od f(a)\op (1-\l)\od f(b)\Big)\\
	&= {1\over 2}\od \big(f(a)\op f(b)\big).
\end{align*}
Thus since $\sigma + \delta = a+b$, it follows that
\begin{equation}\label{e-hh2}
f\left({a+b\over 2}\right) \le_g {1\over 2}\od (f(\sigma) \op f(\delta)) \le_g 	{1\over 2}\od \big(f(a)\op f(b)\big).
\end{equation}
Now, making variable substitution $\l b + (1-\l) a = x$, we find that
\begin{align*}
\int^\op_{[0,1]} f(\sigma)\ot d\l&=	\int^\op_{[0,1]} f(\l b + (1-\l) a)\ot d\l\\
&= \gi\left( \int_0^1 (g\circ f)(\l b + (1-\l) a) \,d\l    \right)\\
&=\gi\left({1\over{b-a}} \int_a^b (g\circ f)(x) \,dx\right)\\
&=\gi\left({1\over{b-a}}g\left(\gi\left( \int_a^b (g\circ f)(x) \,dx\right)\right)\right)\\
&={1\over{b-a}}\od \gi\left( \int_a^b (g\circ f)(x) \,dx\right)\\
&={1\over{b-a}}\od \int^\op_{[a,b]} f(x)\ot dx.
\end{align*}
Consequently, taking $g$-integral throughout \eqref{e-hh2} with respect to $\l$ over $[0,1]$, the inequality \eqref{e-hh1} follows.	
\end{proof}

\begin{corollary}
Let $u,v >_g 0_g,\, u\ne v$ and $g$ be increasing. Then the following inequalities hold:
\begin{equation}\label{e-cor1}
(u\ot v)\pwr{1/2}	\le_g \frac{1}{g(\Ln u)-g(\Ln v)}\od (u\om v) \le_g {1\over 2}\od (u\op v).
\end{equation}
\end{corollary}

\begin{proof}
Clearly, for an increasing function $g$, $f(x)=E^{\gi(x)}$ is pseudo-convex. After some calculations, it can be worked out that by taking $f(x)=E^{\gi(x)}$ in Theorem \ref{t-hh1}, the inequalities \eqref{e-hh1} become
\begin{equation}\label{e-cor2}
\left(E^{\gi(a)}\ot E^{\gi(b)}  \right)	\pwr{1/2} \le_g \frac{1}{b-a}\od \left(E^{\gi(b)}\om E^{\gi(a)}  \right)\le_g \frac{1}{2}\od \left(E^{\gi(a)}\op E^{\gi(b)}  \right) .
\end{equation}
Now, put 
$$
E^{\gi(a)} = u\quad {\rm and}\quad E^{\gi(b)} = v
$$
so that
$$
a=g(\Ln u)\quad {\rm and}\quad b=g(\Ln v).
$$
The inequalities \eqref{e-cor1} now follow with these transformations.
\end{proof}

\begin{remark}
The inequalities \eqref{e-cor1} are the $g$-analogue of the standard geometric-logarithmic-arithmatic mean inequality
$$
\sqrt{uv}\le \frac{u-v}{\ln u-\ln v}\le \frac{u+v}{2},\quad u,v>0,\, u\ne v
$$
which can be obtained by taking the generator $g$ as the identity function in \eqref{e-cor1}
\end{remark}

A refinement of the Hermite-Hadamard inequality has recently been given in (\cite{far}, Theorem 1.1), We prove below its $g$-analogue which is a refinement of the inequality \eqref{e-hh1}.

\begin{theorem}\label{t-hh5}
	Let $f:[a,b]\to \Rg$ be a pseudo-convex function. Then for all $\l\in[0,1]$, the following inequalities hold:
	\begin{equation}\label{e-hh3}
	f\left({a+b\over 2}\right)\le_g \ell(\l) \le_g \left( 1\over{b-a}\right)\od \int^\op_{[a,b]} f(x)\ot dx\le_g L(\l) \le_g {1\over 2}\od \big(f(a)\op f(b)\big),
	\end{equation}
	where
	$$
	\ell(\l):=\l\od f\left({\l b + (2-\l) a\over 2}  \right) \op (1-\l)\od f\left({1+\l b + (1-\l) a\over 2}  \right)
	$$
	and
	$$
	L(\l):={1\over 2}\od\Big( f(\l b + (1-\l)a) \op \l\od f(a) \op (1-\l)\od f(b)\Big).
	$$
\end{theorem}
\begin{proof}
Recall $\sigma=\l b + (1-\l)a$. Clearly $a<\sigma<b$. We apply the inequality \eqref{e-hh1} on the interval $[a,\sigma]$, with $\l\ne 0$ and get
\begin{equation}\label{e-hh4}
f\left({a+\sigma\over 2}\right) \le_g \left( 1\over{\sigma-a}\right)\od \int^\op_{[a,\sigma]} f(x)\ot dx \le_g {1\over 2}\od \big(f(a)\op f(\sigma)\big).	
\end{equation}
Similarly applying again \eqref{e-hh1} on the interval $[\sigma,b]$ with $\l\ne 1$, we get
\begin{equation}\label{e-hh5}
	f\left({\sigma+b\over 2}\right) \le_g \left( 1\over{b-\sigma}\right)\od \int^\op_{[\sigma,b]} f(x)\ot dx \le_g {1\over 2}\od \big(f(\sigma)\op f(b)\big).
\end{equation}

Clearly
\begin{equation}\label{e-hh6}
\l\od f\left({a+\sigma\over 2}\right) \op (1-\l)\od f\left({\sigma+b\over 2}\right) = \ell(\l).
\end{equation}
Also, we find that
\begin{align}
&\l\od \left( 1\over{\sigma-a}\right)\od \int^\op_{[a,\sigma]} f(x)\ot dx \op (1-\l)\od \left( 1\over{b-\sigma}\right)\od \int^\op_{[\sigma,b]} f(x)\ot dx\nonumber\\
&\quad = {1\over{b-a}}\od \left(\int^\op_{[a,\sigma]} f(x)\ot dx \op  \int^\op_{[\sigma,b]} f(x)\ot dx   \right)\nonumber\\
&\quad = {1\over{b-a}}\od \int^\op_{[a,b]} f(x)\ot dx\label{e-hh7}
\end{align}
and
\begin{align}
&\l\od{1\over 2}\od (f(a)\op f(\sigma)) \op 	(1-\l)\od{1\over 2}\od ( f(\sigma)\op f(b)\nonumber\\
&\quad ={1\over 2}\od \Big(\l\od f(a)\op f(\sigma) \op (1-\l)\od f(b)  \Big)\nonumber\\
&\quad =L(\l).\label{e-hh8}
\end{align}
Now, taking the $\od$ product of \eqref{e-hh4} with $\l$ and \eqref{e-hh5} with $1-\l$ and using \eqref{e-hh6}, \eqref{e-hh7}, \eqref{e-hh8}, we obtain
\begin{equation}\label{e-hh9}
	\ell(\l)\le_g {1\over{b-a}}\od \int^\op_{[a,b]} f(x)\ot dx\le_g L(\l). 
\end{equation}
Writing $\displaystyle a+b\over 2$ as
$$
{a+b\over 2} ={\l\over 2}(\l b+(2-\l)a) + \big(\frac{1-\l}{2}\big)(1+\l)b + (1-\l)a
$$
and using the pseudo-convexity, we get
\begin{equation}\label{e-hh10}
f\left({a+b\over 2}\right)\le_e \ell(\l)	
\end{equation}
and from the pseudo-convexity applied on $L(\l)$, we get
\begin{equation}\label{e-hh11}
L(\l)\le_g 	{1\over 2}\od \big(f(a)\op f(b)\big).
\end{equation}
Now, the inequalities \eqref{e-hh3} follow from \eqref{e-hh10} and \eqref{e-hh11}.
\end{proof}

\begin{remark}
Similar to the pseudo-convexity, one can define pseudo-concavity: 	A function $f:[a,b]\to \Rg$ is said to be pseudo-concave on $[a,b]$ if for all $x,y\in[a,b]$ and all $0\le\lambda\le 1$
$$
f(\l x + (1-\l)y)\ge_g \l\od f(x)\op (1-\l)\od f(y).
$$
Theorems \ref{t-hh1} and \ref{t-hh5} can be formulated and proved with pseudo-convexity relaced by pseudo-concavity.
\end{remark}

\section{Concluding Remarks}

In this paper, the classical inequalities, namely, Young's, H\"older's, Minkowski's and Hermite-Hadamard inequalities have been derived in the framework of $g$-calculus. The entire range of index $p\in\R,\,p\ne 0$ has been covered.

\begin{remark}
	In \cite{mes} (see also \cite{mar1}), Mesiar introduced a generalized $g$-integral
	$$
	\int_{[a,b]}^{g,h}f(x)\ot dx = \gi\left(  \int_a^b (g\circ f)(x)h(x)\,dx  \right),
	$$
	where $h$ is a non-negative integrable real function and the corresponding $g$-derivative
$$
D_{g,g} f(x)=\gi\Big((g\circ f)'(x)/h(x)   \Big).
$$
The results of Section 3 can easily be formulated and proved in terms of the above integral.
\end{remark}

\begin{remark}
It is of interest if the inequalities in Section 4 could be obtained for the generalized integrals $\displaystyle\int_{[a,b]}^{g}$ or $\displaystyle\int_{[a,b]}^{g,h}$ instead of $\displaystyle\int_{[a,b]}^\op$.
\end{remark}


\begin{thebibliography}{99}
	
\bibitem {agahi} H. Agahi, Y. Ouyang, R. Mesiar, E. Pap, Endre and M. \v SŠtrboja,  \textit{H\"older and Minkowski type inequalities for pseudo-integral}, Appl. Math. Comput., 217 (2011), 8630--8639.


	
\bibitem{bc} A. Boccuto and D. Candeloro, \textit{Differential calculus in Riesz spaces and applications to g-calculus}, Mediterr. J. Math., 8 (2011), 315--329.

\bibitem{cho} G. Choquet,  \textit{Theory of capacities}, Ann. Inst. Fourier (Grenoble), 5 (1953/54), 131--295.

\bibitem{far} A. El Farissi,  \textit{Simple proof and refinement of Hermite-Hadamard inequality}, J. Math. Inequal., 4 (2010), 365--369.

\bibitem{ich} H. Ichihashi, H. Tanaka and K. Asai, \textit{Fuzzy integrals based on pseudo-additions and multiplications}, J. Math. Anal. Appl., 130 (1988),  354--364.

\bibitem{ling} C.H. Ling, \textit{Representation of associative functions}, Publ. Math. Debrecen, 12 (1965), 189--212.

\bibitem{mar} A. Marko\' v, \textit{A note on g-derivative and g-integral}, Real functions '94 (Liptovský Ján, 1994), Tatra Mt. Math. Publ. 8 (1996), 71--76. 

\bibitem{mar1} A. Marko\' v and B. Rie\v can, \textit{ On the double g-integral}, Novi Sad J. Math. 26 (1996), 161--171.

\bibitem{mas} V.P. Maslov, \textit{Asymptotic methods for solving pseudodifferential equations} (in Russian) "Nauka'', Moscow, 1987.

\bibitem{mes} R. Mesiar, \textit{Pseudo-linear integrals and derivatives based on a generator g}, Real functions '94 (Liptovský Ján, 1994), Tatra Mt. Math. Publ. 8 (1996), 67--70.

\bibitem{pap1} E. Pap, \textit{An integral generated by a decomposable measure}, Zb. Rad. Prirod.-Mat. Fak. Ser. Mat., 20 (1990), 135--144.

\bibitem{pap2} E. Pap, \textit{Decomposable measures and applications on nonlinear partial differential equations} Measure theory (Oberwolfach, 1990). Rend. Circ. Mat. Palermo (2) Suppl. No. 28 (1992), 387–403.

\bibitem{pap3} E. Pap, \textit{The Lebesgue decomposition of the null-additive fuzzy measures}, Zb. Rad. Prirod.-Mat. Fak. Ser. Mat., 24 (1994), 129--137.

\bibitem{pap4} E. Pap, \textit{g-calculus}, Zb. Rad. Prirod.-Mat. Fak. Ser. Mat., 23 (1993),  145--156.

\bibitem{pap5} E. Pap, \textit{Extension of $\op$-decomposable measures}, Atti Sem. Mat. Fis. Univ. Modena, 41 (1993), 109--119.

\bibitem{pap6} E. Pap, M. \v Strboja and I. Rudas, \textit{Pseudo-$L^p$ space and convergence}, Fuzzy Sets and Systems, 238 (2014), 113--128.

\bibitem{flores} H. Rom\' an-Flores, A. Flores-Franuli\v c and Y. Chalco-Cano, \textit{A Jensen type inequality for fuzzy integrals}, Inform. Sci., 177 (2007),  3192--3201.

\bibitem{sug} M. Sugeno, \textit{Theory of fuzzy integrals and its application}, Doctoral dissertation, Tokyo Institute of Technology, 1974.

\bibitem{weber1} S. Weber, \textit{$\perp$-decomposable measures and integrals for Archimedean t-conorms $\perp$}, J. Math. Anal. Appl., 101 (1984), 114--138.

\bibitem{weber2} S. Weber, \textit{Measures of fuzzy sets and measures of fuzziness}, Fuzzy Sets and Systems 13 (1984), 247--271.

\bibitem{zadeh} L.A. Zadeh, \textit{Fuzzy sets}, Inform.  Contr. 8 (1965), 338--353.


	
\end{thebibliography}
\end{document}